\documentclass[12pt,reqno]{amsart}
\usepackage{enumerate}
\usepackage{amsfonts}
\usepackage{color}
\usepackage{amsmath}
\usepackage{amssymb}
\usepackage[latin1]{inputenc}
\usepackage{hyperref}
\usepackage[margin=1.1in]{geometry}

\newcommand{\RR}{{\mathbb R}}
\newcommand{\R}{{\mathbb R}}
\newcommand{\CC}{{\mathbb C}}
\newcommand{\N}{{\mathbb N}}
\newcommand{\NN}{{\mathbb N}}

\newcommand{\hil}{\mathcal H}

\providecommand{\C}[1]{\mathcal{#1}}

\DeclareSymbolFont{bbold}{U}{bbold}{m}{n}
\DeclareSymbolFontAlphabet{\mathbbold}{bbold}

\providecommand{\C}[1]{\mathcal{#1}}
\newcommand{\ee}{\mathcal{E}}
\newcommand{\Dee}{\mathcal{D}}
\newcommand{\Dl}{D_{loc}}

\newtheorem{theorem}{Theorem}

\newtheorem{lemma}{Lemma}[section]
\newtheorem{coro}[lemma]{Corollary}
\newtheorem{definition}[lemma]{Definition}
\newtheorem{prop}[lemma]{Proposition}

 \newlength\headseptemp

\newcommand{\Hmm}[1]{\leavevmode{\marginpar{\tiny%
$\hbox to 0mm{\hspace*{-0.5mm}$\leftarrow$\hss}%
\vcenter{\vrule depth 0.1mm height 0.1mm width \the\marginparwidth}%
\hbox to 0mm{\hss$\rightarrow$\hspace*{-0.5mm}}$\\\relax\raggedright #1}}}
\begin{document}
\title[]{Expansion in generalized eigenfunctions for Laplacians on graphs and metric measure spaces}

\author[]{Daniel Lenz$^1$}
\author[]{Alexander Teplyaev$^2$}
\address{$^1$ Mathematisches Institut, Friedrich Schiller Universit\"at Jena,
  D- 07743 Jena, Germany.
  \href{mailto:daniel.lenz@uni-jena.de}{daniel.lenz@uni-jena.de}
   \url{http://www.analysis-lenz.uni-jena.de/} }
\address{$^2$ Department of Mathematics, University of Connecticut, Storrs, CT 06269-3009, USA, \href{mailto:teplyaev@math.uconn.edu}{alexander.teplyaev@uconn.edu} \url{http://www.math.uconn.edu/~teplyaev}}

\begin{abstract} We consider an arbitrary selfadjoint operator in a separable Hilbert space. To this operator we construct an   expansion in generalized eigenfunctions, in which the original Hilbert space is decomposed as a direct integral of Hilbert spaces consisting of general eigenfunctions. This automatically gives a Plancherel type formula. For suitable operators on metric measure spaces we discuss  some growth restrictions on the generalized eigenfunctions. For Laplacians on locally finite graphs   the generalized eigenfunctions are exactly the solutions of the corresponding difference equation. \tableofcontents
\end{abstract}
\date{\today} %
\subjclass[2010]{Primary:
81Q35, 
05C63, 
28A80; 
Secondary:
31C25, 
60J45, 
05C22, 
31C20, 
35P05, 
39A12, 
47B25, 
58J35, 
81Q10. 
}
\maketitle


\section*{Introduction}
Expansions in generalized eigenfunctions play an important role in the study  of self-adjoint operators. In fact, they are a basic tool in applications  such as  quantum mechanics.  Accordingly, they have received a lot of attention (see e.g. the classical monograph \cite{Ber} of Berezanskii, appendix in \cite{BS}, the corresponding section in \cite{Sim}, or the article  \cite{PSW}).

 The topic  may roughly  be described  as follows:  Let $L$  be a  selfadjoint  operator with spectrum $\varSigma$  on the Hilbert space $\mathcal{H}$. Then, an expansion in generalized eigenfunctions consists of  a measure $\mu$ on $ \varSigma$ and  'projections'  $W_\lambda $ such   that in an appropriate sense $ W_\lambda f$ is a generalized eigenfunction of $L$ to $\lambda$,
$$f = \int_\varSigma W_\lambda f \;  d\mu (\lambda)$$
and
$$ \Phi (L) f = \int_{\Sigma} \Phi (\lambda) W_\lambda f \; d \mu(\lambda)$$
for any $f\in \mathcal{H}$ and any  $\Phi : \varSigma \longrightarrow \CC$ bounded and measurable.
The way this is usually implemented gives
$W_\lambda f$ belonging to some  larger  Hilbert space  than $\mathcal{H}$ \cite{Ber,Sim}.  This larger Hilbert space  arises by considering a  Gelfand triple or another suitable smoothing and is independent of $\lambda$.
This leads to the  issue that the geometry of this larger Hilbert space is in general not compatible with the geometry of $\mathcal{H}$.  In particular, these expansions do not provide a Hilbert space structure on the set of generalized eigenfunctions.

Thus, one may wonder whether actually a stronger type of expansion is available.  This  expansion should consist of
  Hilbert space structures on suitable vector spaces  $\mathcal{H}_\lambda$ which consist of generalized eigenfunctions to the eigenvalue $\lambda$  and
 projections $W_\lambda$ from $\mathcal{H}$ to $\mathcal{H}_\lambda$
such that a decomposition of the form
\begin{itemize}
\item[(D)] $\hspace{15ex} f = \int_\varSigma W_\lambda f d\mu (\lambda)$
\end{itemize}
holds in a weak sense and a Plancherel type formula of the form
\begin{itemize}
\item[(PF)] $\hspace{12ex} \|\Phi (L) f\|^2 = \int_\varSigma |\Phi (\lambda) |^2 \|W_\lambda f \|_{\mathcal{H}_\lambda} d\mu (\lambda)$
    \end{itemize}
is valid for all relevant $f\in \mathcal{H}$ and all sufficiently smooth functions $\Phi$ on $\varSigma$.


Indeed, it is exactly this type of decomposition that   has been needed and obtained in a recent study  by Strichartz / Teplyaev  in a rather specific situation \cite{ST}. The investigations of \cite{ST} provide an eigenfunction expansion for a Laplacian on a specific graph. They  make heavy use of the  situation at hand. In particular, they do not give a clear indication  how  general  selfadjoint operators  general case could be treated. In fact, treatment of such  general operators  is stated as a  problem in \cite{ST}  because of many potential applications to the analysis on fractals (see \cite{HSTZ,HTams,HTjfa,HRT,HKT,IRT,NT,RT,STz,T07} and references therein).
The aim of this paper is to provide a solution to this problem.

Our basic idea is to  use   direct integral theory to exhibit $\mathcal{H}$ as
$$\mathcal{H} = \int_\varSigma^\oplus \mathcal{H}_\lambda d\mu$$
with $\mathcal{H}_\lambda$ consisting of generalized eigenfunctions of $L$ to the eigenvalue $\lambda$. Let us stress that a direct integral decomposition of $\mathcal{H}$ giving a Plancherel type formula is a more or less direct consequence of the spectral theorem (see e.g. the appendix in \cite{BS}). Our main achievement in this paper is to exhibit the fibers as spaces of generalized eigenfunction. Put differently, our main achievement is  to equip certain  spaces of generalized eigenfunctions with a Hilbert space structure.

To do so, we first introduce  a suitable concept of generalized eigenfunction  and  then combine this  with the considerations concerning  generalized eigenfunction expansions developed by Poerschke / Stolz / Weidmann \cite{PSW}. These considerations in themselves do not give vector spaces of generalized eigenfunctions (let alone Hilbert spaces of generalized eigenfunctions). However, they can be read to provide a basis for such spaces. With our concept of generalized eigenfunction, we  can then make these bases into orthonormal basis and this is how the Hilbert space structure is introduced.

 Our class of generalized eigenfunctions is \emph{a priori} a bit weaker than notions used so far. On the other hand, for Laplacians on graphs, we recover the canonical notion of solution. So, we obtain a  convincing theory
 for these  operators (which are the starting point of our investigations as discussed  above).

As a byproduct of our discussion we even  obtain  uniqueness of the fibers in an appropriate sense.


The paper is structured as follows:   In Section \ref{Locally} we introduce our main class of  examples viz locally finite operators on discrete spaces and state the expansion theorem for these operators.
In Section \ref{general} we then  provide an abstract result on expansion in generalized eigenfunctions valid for arbitrary selfadjoint operators on a separable Hilbert space. A short discussion  of \emph{a priori} 'growth restrictions' on generalized eigenfunctions is given in Section \ref{growth}.
In Section \ref{Metric} we discuss certain finer  growth  properties of the generalized eigenfunctions in the case of (metric) measure spaces.
In Section \ref{Getting} we then return to discrete spaces and discuss how the abstract theory developed earlier gives a proof of the expansion theorem in this case.
Finally, the appendix contains a short discussion of the direct integral theory needed to understand the paper.


Throughout the paper, all inner product and dual pairings are assumed to be linear in the second argument and antilinear in the first argument.

\medskip

\noindent\textbf{Acknowledgments.}  A substantial part  of this work was done while the authors were visiting the Isaac Newton Institute for the Analysis on Graphs and its Applications Follow-up Meeting in 2010. The stimulating atmosphere of this conference is gratefully acknowledged.

D.L. gratefully acknowledges partial support by the German Research Foundation (DFG)  as well as enlightening discussions with Gunter Stolz and Peter Stollmann. He also takes this opportunity to thank the Mathematics departments of the University of Lyon and of the University of Geneva  for hospitality.

A.T. is deeply thankful to Peter Kuchment and Robert Strichartz for interesting and helpful discussions related to this work. His research is supported in part by NSF grant DMS-0505622.

\section{Locally finite operators on discrete measure spaces}\label{Locally}
In this section we introduce the main example of our interest viz locally finite operators on discrete measure spaces. Let us emphasize that this includes the usual Laplacians on locally finite graphs.


Let $V$ be a discrete finite or countably infinite  set and  $m$  a measure on $V$ with full support (i.e. $m$ is a map on $V$ taking values in $(0,\infty)$). We  then call $(V,m)$ a \textit{discrete measure space}. The set of functions on $V$ with finite support is denoted by $C_c (V)$. The set of all functions on $V$ is denoted by $C (V)$.
The corresponding Hilbert space
$$\ell^2 (V,m) :=\{ u : V\longrightarrow \CC : \sum_{x\in V} |u(x)|^2 m(x) < \infty\}$$
is equipped with the inner product
$$\langle v, u \rangle := \sum_{x\in V} \overline{v(x)} u(x) m(x).$$


To each $a : V\times V\longrightarrow \CC$ we can associate the formal operator $\widetilde{A}$ mapping  the vector space
$$D (a) :=\{ w\in  C(V) : \sum_{y\in V} |a(x,y) w(y) | m(y) <\infty \;\mbox{for all $x\in V$}\}$$
to the vector space  $C(V)$ via
$$(\widetilde{A} w)(x) := \sum_{y\in V} a(x,y) w(y) m(y).$$
An  operator $A$  on $\ell^2 (V,m)$ is then said to have the  \textit{kernel}  $a $ if $A$ is a restriction of $\widetilde{A}$. If the domain of $A$ contains $C_c (V)$,  the kernel of $A$ is uniquely determined.
A  function $a : V\times V\longrightarrow \CC $ is   called \textit{locally finite} if the set
$$ \{y \in V: a(x,y)\neq 0\}$$
is finite for any $x\in V$.  In this case $D(a) = C(V)$ holds.
An operator $A$ on $\ell^2 (V,m)$ is called \textit{locally finite}, if $C_c (V)$ is contained in the domain of definition of $A$ and $A$ has a locally finite kernel $a$.  For such operators, we say that a function $\varphi \in C(V)$ is a \textit{generalized eigenfunction of $A$ to the eigenvalue $\lambda\in \CC$}  if $\varphi$ satisfies
$$ (\widetilde{A} - \lambda) \varphi \equiv 0.$$


Note that we have a dual pairing between $C(V)$ and  $C_c (V)$  given by  the  sesquilinear map  $$(\cdot,\cdot)_m: C (V)\times C_c  (V) \longrightarrow \CC$$ defined by
$$ (g,u)_m:= \sum_{x\in V} \overline{ g(x) } u  (x) m(x).$$

Moreover, for $\omega : V\longrightarrow [0,\infty)$ define
$$C_\omega (V):=\{u\in C(V) : \sum_{x\in V} \omega (x)^2 |u(x)|^2 m(x) < \infty\}.$$

Our main result in this context reads as follows (see the appendix for some short review of direct integral theory).

\begin{theorem}\label{main-concrete}
Let $(V,m)$ be a discrete measure space. Let $L$ be locally finite
and selfadjoint with spectrum $\varSigma$.   Then, there exist a
measurable family of Hilbert spaces  $\mathcal{H}_\lambda$,
$\lambda\in \varSigma$, a measure $\mu$ on $\varSigma$, a unitary
map
$$ W : \ell^2 (V,m)\longrightarrow \int^\oplus_\varSigma \mathcal{H}_\lambda d\mu(\lambda)$$
and a family  of linear operators $W_\lambda : C_c (V)\longrightarrow \mathcal{H}_\lambda$, $\lambda\in \varSigma$, such that the following holds:
\begin{itemize}
\item
For each $\lambda\in \varSigma$, the space  $\mathcal{H}_\lambda$ is a subspace of the vector space of generalized  eigenfunctions of $L$ to the eigenvalues $\lambda$.
\item For all $u\in \ell^2 (V,m)$,  $g\in C_c (V)$ and measurable bounded  $\Phi : \varSigma \longrightarrow \CC$   the equality
$$ \langle \Phi (L)  u , g\rangle = \int_\varSigma  ( \Phi (\lambda) W u (\lambda), g )_m  d\mu (\lambda)$$
holds.
\item For all $g\in C_c (V)$ the functions  $\lambda \mapsto (W g) (\lambda )$ and $ \lambda\mapsto W_\lambda g$ agree  $\mu$-almost everywhere.
\item $W  \Phi (L) W^{-1} = M_{\Phi}$ for any measurable bounded $\Phi : \varSigma \longrightarrow \CC$.
Here, $M_{\Phi }$ is the operator acting in the fiber $\mathcal{H}_\lambda$ by multiplication by $\Phi(\lambda)$.
\end{itemize}
The Hilbert spaces  $\mathcal{H}_\lambda$ are uniquely determined by the above properties up to changes on sets of $\lambda$ with zero $\mu$-measure. Moreover, for any
 $\omega : V\longrightarrow [0,\infty)$  with  $\sum_{x\in V} \omega (x)^2 <\infty$  the inclusion
 $$\mathcal{H}_\lambda \subset C_\omega (V)$$
 holds for $\mu$-almost every $\lambda\in\varSigma$.
\end{theorem}

The result will be derived from the abstract expansion result of the next section and some auxiliary considerations. Here, we note the following immediate corollary.

\begin{coro} Let the situation of the theorem be given. Then, there exists a map $w : \varSigma \times V\times V\longrightarrow \CC$ such that  $ x\mapsto w (\lambda, x,y)$ is a generalized eigenfunction to $\lambda$ for each $y\in V$ and
$$ \langle g,  \Phi (L) f\rangle  = \int_\varSigma   \Phi (\lambda)  \sum_{x,y\in V}   \overline{g(y)} w(\lambda,y,x) f(x)  d\mu (\lambda)$$
for all $f,g\in C_c (V)$ and $\Phi : \varSigma \longrightarrow \CC$ bounded and measurable.
\end{coro}

\noindent\textbf{Remarks} (a)  A particularly important class of
operators covered by these results are the 'usual' Laplacians
encountered on locally finite graphs. Note that spectral theory of
these Laplacians has attracted a lot of attention  in recent years.
In fact, starting with  \cite{Woj2,Jor} the case of Laplacians,
which are unbounded, has become a focus of intensive research, see
e.g. \cite{KL,KL2,HKLW,CdVTHT,Hua,HKMW,HKW,Fol,KLVW,Web,Woj} and
references therein. The above theorem should be helpful for further
studies.

(b) Note that the uniqueness statement of the theorem gives that the eigenfunction expansion provided in the special situation of \cite{ST} must actually agree with the eigenfunction expansion given above.

\section{A general expansion result}\label{general}
In this section we present and prove an abstract  result on expansion in generalized eigenfunctions. The proof of our result relies on direct integral theory and the eigenfunction expansion presented in \cite{PSW}. The main novelty is  to introduce a notion of generalized eigenfunction weak enough so that the corresponding terms of \cite{PSW} can be interpreted as belonging to some Hilbert space of generalized eigenfunctions. Unlike the direct integral decompositions based on spectral theorem alone our expansions in eigenfunctions turn out to be unique.


Let $L$ be a self-adjoint operator on a separable Hilbert space
$\mathcal{H}$.  In order to simplify the statement of our theorem we
will introduce various pieces of notation next.

A measure $\mu$ on $\RR$ is called a \textit{spectral measure} for
$L$ if for a Borel set $A\subset \R$ the equality $\mu (A) = 0$
holds if and only if $1_A (L) = 0$ (where $1_A$ is the
characteristic function of $A$). It is a standard fact on
self-adjoint operators  on a separable Hilbert space that spectral
measures exist.

A pair $(W, (\mathcal{H}_\lambda)_{\lambda\in \varSigma})$ consisting of a measurable family of Hilbert spaces $((\mathcal{H}_\lambda)_{\lambda\in \varSigma})$ and a unitary map $W : \mathcal{H}\longrightarrow \int_\varSigma^\oplus \mathcal{H}_\lambda d\mu (\lambda)$  is called a \textit{direct integral decomposition} of $L$  with respect to the spectral measure  $\mu$ if
$$ W L W^{-1} = \int^\oplus_\varSigma M_\lambda d\mu (\lambda)$$
holds,
where   $M_{\lambda}  : \mathcal{H}_\lambda \longrightarrow \mathcal{H}_\lambda, \phi \mapsto  \lambda \phi$.

As is well known (and in fact a direct consequence of existence of
ordered spectral representations  as given in Lemma 2 (c) of
\cite{PSW}), there exist direct integral decompositions of a
selfadjoint operator $L$ with respect to any spectral measure $\mu$.

Whenever $(W, (\mathcal{H}_\lambda)_{\lambda\in \varSigma})$ is a direct integral decomposition  of $L$, then
$$ Wf = \int_\varSigma^\oplus W f (\lambda) d\mu (\lambda)$$
holds for all $f\in \mathcal{H}$.  Moreover, one then  obtains from basic  direct integral theory (see Appendix)
 $$W \Phi (L) W^{-1} = M_\Phi$$
 for all    measurable $\Phi : \varSigma \longrightarrow \CC$. Thus, in particular, a   Plancherel formula of the form
 $$ \|\Phi (L) f\|^2 = \int_\varSigma |\Phi (\lambda)|^2 \langle W f (\lambda), W f (\lambda) \rangle_{\mathcal{H}_\lambda} d\mu (\lambda)$$
is   valid for all  $f\in \mathcal{H}$ and all bounded measurable  functions $\Phi$ on $\varSigma$.

The preceding discussion shows that in a certain sense any  direct integral decomposition has the features  (D) and (PF)
discussed in the introduction.
Our main task is to find a direct integral decomposition with fibers consisting of Hilbert spaces of generalized eigenfunctions. This will be achieved with the concept of generalized eigenfunction introduced next.

Let $\mathcal{D}$ be any linear subspace of $\mathcal{H}$.
By  the algebraic dual of $\mathcal{D}$ we understand the space of linear functions on $\mathcal{D}$, which are not necessarily continuous.
An element $\varphi$ of  the algebraic dual of $\mathcal{D}$ is called a \textit{$\mathcal{D}$- eigenfunction of $L$} to the eigenvalue $\lambda\in \RR$ if
\begin{equation}\label{eq-duaL}
    (\varphi, L u ) = \lambda (\varphi, u  )
\end{equation}
for all $u\in \mathcal{D} \cap D(L)$ with $L u\in \mathcal{D}$.   Here, $(\cdot, \cdot)$ denotes the canonical dual pairing between the algebraic dual of $\mathcal{D}$ and $\mathcal{D}$ itself.
For later application,  it will be convenient to have $(\cdot,\cdot)$ as  sesquilinear  form. We achieve this by  defining the scalar multiplication on the algebraic dual of  $\mathcal{D}$ via
$$(\alpha \cdot \varphi) ( v) := \overline{\alpha} \varphi (v)$$
for $\alpha \in \CC$, $v\in \mathcal{D}$ and $\varphi$ in the algebraic dual of $\mathcal{D}$.

After these preparations we can now state our main abstract theorem.

\begin{theorem}\label{main-abstract}
Let $\mathcal{H}$ be a separable Hilbert space and $L$ a selfadjoint
operator on $\mathcal{H}$ with spectrum $\varSigma$. Let $\mu$ be a
spectral measure of $L$.  Let $\mathcal{D}$ be an arbitrary dense
subspace of $\mathcal{H}$ admitting a countable algebraic basis.
Then, there exists a direct integral decomposition $(W,
(\mathcal{H}_\lambda)_{\lambda\in \varSigma})$ with respect to $\mu$
such that  the following holds:
\begin{itemize}
\item
For each $\lambda\in \varSigma$, the vector space
 $\mathcal{H}_\lambda$ is a subspace of the space of
 $\mathcal{D}$-eigenfunctions of $L$ to the eigenvalue $\lambda$.

\item For all $f\in \mathcal{H}$ and $v\in \mathcal{D}$ and any measurable bounded $\Phi : \varSigma \longrightarrow \CC$  the equality
    $$\langle \Phi (L) f, v\rangle = \int_\Sigma ( \Phi (\lambda)  W f(\lambda), v) d\mu (\lambda)$$
    is valid.
\end{itemize}
Such a direct integral decomposition is unique up to changing the
$\mathcal{H}_\lambda$ and $W$ on a set of $\lambda$'s with zero
$\mu$-measure.
\end{theorem}

\noindent\textbf{Remarks.} (a)   A  short interpretation of the two  points appearing in the  theorem may be given as follows:
The first  point underlines that $\mathcal{H}_\lambda$ consists of generalized eigenfunctions of $L$. Given the first point, the main content of the
 second point is an  expansion of the form  $f = \int_\varSigma W f (\lambda) d\mu (\lambda)$ in the weak sense.

(b)  We require the existence of a countable algebraic basis of $\mathcal{D}$.
While this  is a somewhat  strong requirement it does not preclude the (arguably) most natural choices of $\mathcal{D}$ viz as a core of $L$. In fact, whenever $\mathcal{C}$ is a core for $L$ (i.e. the graph of $L$ is the closure of the graph of the restriction of $L$ to $\mathcal{C}$) then one can find a countable dense set in $\mathcal{C}$ determining $L$. The linear span of this set will then have a countable basis and be a core for $L$ (contained in $\mathcal{C}$).

(c)  For Laplacians on graphs there is a most natural choice of $\mathcal{D}$ as the set of functions with finite support.  This yields the canonical notion of generalized solution (see below for further discussion).

\medskip

As $\mathcal{D}$ has a countable algebraic  basis we immediately obtain  the following corollary
 of this theorem.

\begin{coro} Assume the situation of  Theorem~\ref{main-abstract}. Then,  there exists
 a family of linear operators $W_\lambda : \mathcal{D} \longrightarrow \mathcal{H}_\lambda$, $\lambda\in \varSigma$ with  $(W v)  = \lambda \mapsto W_\lambda v$
 for all $v\in \mathcal{D}$. In particular,

$$ \langle g,  \Phi (L) f \rangle  = \int_\varSigma \Phi(\lambda)  ( W_\lambda g, f) d\mu (\lambda)$$
holds for all $f,g\in \mathcal{D}$ and $\Phi : \varSigma \longrightarrow \CC$ bounded and measurable.
\end{coro}

\noindent\textbf{Remark.} Let us note that the theorem and its corollary indeed provide an expansion with the properties aimed at in the introduction.

\medskip

The remaining part of this section is devoted to the  \textit{Proof of Theorem \ref{main-abstract}}:


We will first show the uniqueness statement:  suppose $(W^{(1)}, \mathcal{H}^{(1)}_\lambda)$ and $(W^{(2)},  \mathcal{H}^{(2)}_\lambda)$ are decompositions with the desired properties,
and $v_1,v_2,\ldots$ be a countable basis of $\mathcal{D}$.

As $\mathcal{D}$ is dense in $\mathcal{H}$ with a countable basis  and $W^{(1)}$ and $W^{(2)}$ are  unitary maps, we have  for $\mu$-almost every $\lambda\in\varSigma$
\begin{equation} \label{dense}
\mathcal{H}^{(j)}_\lambda = \overline{Lin\{  W^{(j)}  v_k (\lambda): k\in \NN\}   }, \: j =1,2
\end{equation}
where $Lin$ means the set of linear combinations of the given set of vectors.

Now, fix $f\in \mathcal{H}$ and $v\in\mathcal{D}$.
As
$$\langle \Phi (L) f, v\rangle = \int_\Sigma \overline{\Phi (\lambda)} (   W^{(j)} f(\lambda), v) d\mu (\lambda)$$
holds for all bounded measurable $\Phi$ on $\varSigma$  for $j=1,2$,
we infer that
$$(W^{(1)}  f (\lambda), v) = (W^{(2)} f (\lambda), v)$$
 for $\mu$-almost every $\lambda$.
 As $\mathcal{D}$ has a countable basis, this implies that for any fixed $f$
 $$W^{(1)}  f (\lambda) = W^{(2)} f (\lambda)$$
 for $\mu$-almost every $\lambda$.  This gives  for $\mu$-almost every $\lambda$
\begin{equation} \label{agree}
W^{(1)}  v_k (\lambda) = W^{(2)} v_k (\lambda)
\end{equation}
for all $k\in \NN$.
Moreover, as the $W^{(j)}$ are decompositions we find for fixed $f,g\in \mathcal{H}$
{
\begin{eqnarray*}\int_\varSigma \Phi (\lambda) \langle W^{(1)} f (\lambda), W^{(1)}  g (\lambda)\rangle_{H^{(1)}_\lambda} d\mu & = & \langle f, \Phi (L) g\rangle\\
& =&
\int_\varSigma \Phi (\lambda) \langle W^{(2)}  f (\lambda), W^{(2)}  g (\lambda)\rangle_{H^{(2)}_\lambda} d\mu
\end{eqnarray*}
}
for all bounded measurable $\Phi : \Sigma\longrightarrow \CC$. This in turn shows
$$
\langle W^{(1)} f (\lambda), W^{(1)}  g (\lambda)\rangle_{H^{(1)}_\lambda} = \langle W^{(2)}  f (\lambda), W^{(2)}  g (\lambda)\rangle_{H^{(2)}_\lambda}$$
for $\mu$-almost every $\lambda$.  Thus, we obtain  for $\mu$-almost every $\lambda$
\begin{equation} \label{drei}\langle W^{(1)} v_j (\lambda), W^{(1)}  v_k (\lambda)\rangle_{H^{(1)}_\lambda} = \langle W^{(2)}  v_j (\lambda), W^{(2)}  v_k (\lambda)\rangle_{H^{(2)}_\lambda}
\end{equation}
for all $k,j\in \NN$.


Now, the uniqueness statement follows from \eqref{dense}, \eqref{agree} and \eqref{drei}.


We now turn to proving the existence statement.  This will be done in two steps. In the first step we will recall the setting of \cite{PSW} and its  main abstract result on expansion in generalized eigenfunctions. In the second step we will  then revise this result to derive the theorem.


We start with the first step.  To any   self-adjoint operator $T\geq 1$ on $\mathcal{H}$   we can associate  the following two
auxiliary Hilbert spaces:
\begin{equation}\label{eq-h-p}
    \hil_+:= \hil_+(T) :=D(T) \:\mbox{with}\;\: \langle x,y
\rangle_+:=\langle Tx,Ty \rangle
\end{equation}
 and
 $$\hil_-:= \mbox{ completion of $\hil$ w.r.t.  the scalar product $\langle x,y\rangle_-:=\langle T^{-1}x,T^{-1}y
\rangle$}.$$
Then, the inner product on $\hil$ can be naturally extended to
a dual pairing
$$ \langle \cdot, \cdot \rangle_d : \hil_-\times \hil_+ \longrightarrow \CC. $$

 Let now  $\mu$ be  a spectral measure for $L$. Then, there exists an \textit{ordered spectral representation}  to $\mu$ i.e. a
sequence of subsets $M_j\subset \RR$, such that $M_j \supset
M_{j+1}$ together with a unitary map $U$
\begin{equation*}
 U=(U_j):\hil\to \oplus_{j=1}^N L^2(M_j,d\mu)
\end{equation*}
with the intertwining property
\begin{equation}\label{intertwine}
 U\Phi(L)=M_\Phi U,
\end{equation}
for every measurable function $\Phi$ on $\RR$. The  index $j$ takes values in a  countable  set which we assume to be given by $1,\ldots, N$ with $N=\infty$ allowed.
If  $\gamma : \RR
\longrightarrow \CC$ is   continuous and bounded with $|\gamma|>0$ on $\varSigma
$ and $T\geq 1$ is
such that $\gamma(L)T^{-1}$ is a Hilbert-Schmidt operator, then the following holds by
the main abstract result of \cite{PSW}:  There exist
 measurable functions
 $$\varphi_{j}:M_j\to\hil_-, \lambda\mapsto \varphi_{j} (\lambda),$$
  for
$j=1,\ldots,N$ such that the following properties hold:
\begin{itemize}
\item[(a)]\label{ia} $
U_j f(\lambda)=\langle \varphi_{j}(\lambda), f \rangle_d \text{ for } f\in \hil_+ \text{ and }\mu\text{-a. e.  }\lambda \in M_j.  $
\item[(b)] For every $g=(g_j)\in \oplus_j L^2(M_j,d\mu)$
\begin{equation*}
U^{-1} g=\lim_{n\to N, E\to \infty} \sum\limits_{j=1}^n \int\limits_{M_j\cap [-E,E]}g_j(\lambda)\varphi_{j} (\lambda) d\mu(\lambda)
\end{equation*}
and, for every $f\in \hil$,
\begin{equation*}
f=\lim_{n\to N, E\to \infty} \sum\limits_{j=1}^n \int\limits_{M_j\cap [-E,E]} (U_j f)(\lambda) \varphi_{j} (\lambda) d\mu(\lambda).
\end{equation*}
(Here, the integrals exist as elements of $\mathcal{H}$ and  limits are meant in the sense of convergence in  $\mathcal{H}$.)
\item[(c)]For each $f\in \{ g\in D(L)\cap\hil_+ | \ Lg\in \hil_+\}$ we have for any $j$ and $\mu$ - almost every $\lambda$
\begin{equation*}
\langle \varphi_{j} (\lambda), L f \rangle_d =\lambda \langle \varphi_{j} (\lambda), f \rangle_d .
\end{equation*}

\end{itemize}
Note that (a) and (b)  deal with properties of $U$ while (c) gives a weak version of $\varphi_{j}(\lambda)$ being an eigenfunction.


We will now turn to the second step and reformulate the above expansion.

Choose $\gamma : \RR
\longrightarrow \CC$    continuous and bounded with $|\gamma|>0$ on $\varSigma
$ and $T\geq 1$
such that

\begin{itemize}
\item $\gamma(L)T^{-1}$ is a Hilbert-Schmidt operator and

 \item  $\mathcal{D}$ is contained in the domain of definition of $T$.
\end{itemize}

Such a choice is always possible: In fact, it suffices to choose  $1\geq \omega_n>0$, $n\in\NN$ with $\sum_n |\omega_n|^2 < \infty$ and an orthonormal basis $e_n$, $n\in \NN$,  contained in $\mathcal{D}$ and to let  $S$ be the unique linear operator with $ S e_n = \omega_n e_n$, $n\in\NN$. Then, $S$ is Hilbert-Schmidt and invertible. Hence, $T := S^{-1}$ exists  and $\gamma (L) T^{-1} = \gamma (L) S$ is Hilbert-Schmidt for any bounded function $\gamma$ (see Section \ref{growth} for further exploration of this situation).

This means that we are indeed in a position to apply the main result of \cite{PSW}, which was just discussed.

In order to avoid some tedious but non-essential technicalities, we will assume without loss of generality that the arising  index set  $J$ equals to $ \N$ and  that
$$ \varSigma = M_j$$
for all $j\in J$. (This will save us from having to deals with families of Hilbert spaces with varying dimension.)

It will be convenient to introduce to each $f\in \hil_+$ and $\lambda\in \varSigma$ the function
$$e_f (\lambda) : J \longrightarrow \CC,\:\;  e_f (\lambda) (j)  := \langle \varphi_j (\lambda), f \rangle_d.$$
We will proceed by a series of claims.

\medskip

\textit{Claim 1.} For any  $f\in \hil_+$ we have
 $$\int_\varSigma \sum_{j\in J}   | e_f (\lambda) (j)|^2 d\mu (\lambda) = \|f\|^2 < \infty.$$
  In particular, for each $f \in \hil_+ $
$$ \sum_{j\in J}| e_f (\lambda) (j)|^2<\infty$$
holds for $\mu$-almost every $\lambda\in \varSigma$.


Proof of the claim.  By (a), we have  for $f\in \hil_+$ that  $(U f)_j (\cdot) = \langle \varphi_j (\cdot), f \rangle_d = e_f (\cdot) (j)$. Thus, we can calculate
\begin{eqnarray*}
\|f\|^2 &=& \sum_{j\in J} \| (U f)_j\|^2_{L^2 (\varSigma,\mu)}\\
&=& \sum_{j\in J} \int_\varSigma |(U f)_j (\lambda)|^2 d\mu (\lambda)\\
&=&  \int_\varSigma \sum_{j\in J } |\langle \varphi_j (\lambda), f \rangle_d|^2 d\mu (\lambda)
\end{eqnarray*}
and we obtain the claim.

\medskip

From the claim we can conclude that for each $f\in \hil_+$ the function
$e_f (\lambda) : J \longrightarrow \CC$
belongs to $\ell^2 (J)$ for  $\mu$-almost every $\lambda$.
In fact, more is true. To discuss  this, let $\mathcal{D}$ be any dense subspace of $\hil_+$ with a countable basis. By the countability of the basis, we can find a subset   of $\varSigma$  of full $\mu$ - measure such that for any $\lambda$ in this subset and any $v\in \mathcal{D}$, the element $e_v (\lambda)$ belongs to $\ell^2 (J)$.  For such $\lambda$, we define   $K_\lambda$ to be the subspace of  $\ell^2 (J)$  generated by $\{ e_v (\lambda) : v\in \mathcal{\mathcal{D}} \}$. For all other $\lambda$ we define $K_\lambda$ to be $\{0\}$.

\medskip

\textit{Claim 2.} For $\mu$-almost every $\lambda$  the space $K_\lambda$ equals  $\ell^2 (J)$.


Proof of the claim. By construction the set $e_v (\lambda)$, $v\in \mathcal{D}$, has the following two properties:
\begin{itemize}
\item  Its span is dense in $K_\lambda$ for  $\mu$- almost-every $\lambda$.
\item The map $\lambda \mapsto \langle e_v (\lambda), e_w (\lambda)\rangle = \sum_j  (U_j v) (\lambda) \overline{ (U_j w) (\lambda)}$  is measurable for any $v,w\in \mathcal{V}$.
\end{itemize}
Thus, the family $(K_\lambda)$ is a measurable family of Hilbert spaces. Accordingly, the orthogonal complement  $K_\lambda^\perp$ of $K_\lambda$ in $\ell^2 (J)$ also form a measurable family of Hilbert spaces.

Assume now that $K_\lambda^\perp \neq\{0\}$ for a set of $\lambda$ of positive $\mu$ measure.  Then, we can find a  $c \in \int_\varSigma^\oplus K_\lambda^\perp d\mu (\lambda)$  with $\|c\|\neq 0$, i.e. a function $c$ on  $\varSigma$ with
\begin{equation}
\label{perp}
 c( \lambda)\in K_\lambda^\perp
 \end{equation}
for every $\lambda\in \varSigma$ and
$$ 0 < \int_\varSigma  \sum_{j\in J} |c (\lambda) (j)|^2 d\mu (\lambda) < \infty.$$
Consider now $g = (g_j) \in \oplus_j L^2 (\varSigma,\mu)$ with
$$ g_j (\lambda) = c(\lambda)(j).$$
Then, we have
$U^{-1} g \equiv 0$ as $\mathcal{D}$ is dense in $\mathcal{H}$ and  for each $v \in \mathcal{D}$ we can calculate by (b)
 \begin{eqnarray*}  \langle(U^{-1} g), v \rangle  &\stackrel{(b)}{=}& \lim_{n\to \infty, E\to \infty} \sum_{j=1}^n \int_{\varSigma \cap [-E,E]} \langle  g_j (\lambda)  \varphi_j (\lambda), v\rangle_d  d\mu (\lambda)\\
&=& \lim_{E\to \infty, n\to \infty } \int_{\varSigma \cap [-E,E]}  \sum_{j=1}^n  \overline{ g_j (\lambda)}  e_v (\lambda) (j)  d\mu (\lambda)\\
\;\: &\stackrel{(Claim \: 1)}{=} &  \lim_{E\to \infty } \int_{\varSigma \cap [-E,E]}  \sum_{j=1}^\infty  \overline{ g_j (\lambda)}  e_v (\lambda) (j)  d\mu (\lambda)\\
&=& \lim_{E\to \infty} \int_{\varSigma \cap [-E,E]} \langle c(\lambda), e_v (\lambda)\rangle_{\ell^2 (J)}  d\mu (\lambda)
 \ \ \stackrel{\eqref{perp}}{=}  0.
\end{eqnarray*}

As $U$ is unitary,  we infer that $g\equiv 0$. On the other hand, we have
 $$\|g\|^2 = \int_\varSigma \sum_j |c (\lambda) (j)|^2 d\mu (\lambda) >0.$$
This contradiction shows that  $K_\lambda^\perp = \{0\}$ for $\mu$-almost every $\lambda$.
This finishes the proof of the claim.

\medskip

Let now $\varSigma_1$ be the set of  $\lambda\in \varSigma $ for which both  the conclusions of Claim 1 apply  for all $v\in \mathcal{D}$ and the conclusion of  Claim 2 holds. Then, $\varSigma_1$ has full $\mu$-measure (as $\mathcal{D}$ has a countable basis).

As $\mathcal{D}$ has a countable basis, so has its subspace  $\mathcal{D}^{\widetilde{}}$ consisting of  all $v\in \mathcal{D}$ such that $L v $ exists and belongs to $\mathcal{D}$. Thus, by (c), we can find a set $\varSigma_2$ of full $\mu$-measure such that for all $\lambda\in \varSigma_2$ we have
\begin{equation}\label{wef} \lambda \langle \varphi_j (\lambda), v \rangle_d = \langle  \varphi_j (\lambda), L v \rangle_d
\end{equation}
for all $j\in J$ and $v\in \mathcal{D}^{\widetilde{}}$. Let $\varSigma_0 := \varSigma_1\cap \varSigma_2$.
This set has then again  full $\mu$-measure. For  $\lambda\in \varSigma_0$ we define $\mathcal{H}_\lambda$ to be  the vector space of all $\varphi$ in the algebraic dual space of
$\mathcal{D}$   which can be written in the form
$$ \varphi = \sum_j a_j \varphi_j (\lambda)  $$
with $(a_j)\in \ell^2 (J)$. Here, the sum belongs indeed to the algebraic dual space as  for each $v\in \mathcal{D}$ we have absolute convergence (and hence existence) of
$$ (\varphi, v) = \sum_{j}  \overline{a_j} \langle \varphi_j (\lambda), v \rangle_d,$$
by Claim 1. Then, by this pointwise existence and \eqref{wef},   we obtain that each $v\in \mathcal{H}_\lambda$ is indeed a generalized eigenfunction of $L$. Moreover, Claim 2 easily gives that the map
$$ j_\lambda :  \ell^2 (J)\longrightarrow \mathcal{H}_\lambda,\;\:  j_\lambda (a) :=\sum_{j\in J} a_j \varphi_j (\lambda), $$
is injective.  By construction, the map $j_\lambda$ is also surjective. Thus, we can identify $\ell^2 (J)$ with $\mathcal{H}_\lambda$. In this way, $\mathcal{H}_\lambda$ becomes a Hilbert space consisting of generalized eigenfunctions of $L$.


For the remaining $\lambda\in \varSigma \setminus \varSigma_0$ we define $\mathcal{H}_\lambda:=\{0\}$. Moreover, we redefine each  $\varphi_j$ on $\varSigma \setminus \varSigma_0$ by setting it zero there.


By construction the map
$$  S  : \bigoplus_{j\in J} L^2 (\varSigma,\mu) \longrightarrow \int^\oplus_\varSigma \mathcal{H}_\lambda d\mu (\lambda),\;\:  S (g_j) (\lambda)  :=\sum_{j} g_j (\lambda) \varphi_j (\lambda),$$
is then unitary. Moreover, it is not hard to see that
$$ S M_\Phi = M_\Phi S$$
for any measurable function $\Phi$ on $\R$.Thus,
$$ W := S \circ U : \mathcal{H} \longrightarrow \int^\oplus_\varSigma \mathcal{H}_\lambda d\mu (\lambda)$$
is a unitary map as well and by the preceding equality and \eqref{intertwine} we obtain
$$ W \Phi (H) = S U \Phi (H) = S M_\Phi U = M_\Phi  S U = M_\Phi W$$
for any measurable $ \Phi$ on $\RR$. Now, a direct calculation  shows that
$$ (W f) (\lambda) = \sum_{j\in J}  (U_j f)(\lambda) \varphi_j (\lambda)$$
for $\mu$-almost every $\lambda\in \varSigma$.  Replacing $f$ by $\Phi (L) f$ if $f\in D(\Phi (L))$ for $\Phi : \varSigma \longrightarrow \CC$,
we find    from \eqref{intertwine}
   \begin{equation}\label{calculating-w}
   \Phi (\lambda) (W  f) (\lambda) = \sum_{j\in J}  (U_j \Phi (L)  f)(\lambda) \varphi_j (\lambda)
   \end{equation}

As $\mathcal{D}$ has a countable basis,
this allows us to define linear  maps $W_\lambda : \mathcal{D}\longrightarrow \mathcal{H}_\lambda$ with
$$ (W \Phi (L)  f) (\lambda) = \Phi (L)  W_\lambda f$$
for $\mu$-almost every $\lambda\in \varSigma$ and any bounded measurable $\Phi : \varSigma \longrightarrow \CC$.
Moreover, for each $f\in \mathcal{H}$ and $v\in \mathcal{D}$ and
$\Phi :\varSigma \longrightarrow \CC$ measurable with $f\in D(\Phi (L))$
we have  by (b) above
\begin{eqnarray*}
\langle \Phi (L) f, v\rangle  & \stackrel{(b)}{=}&  \lim_{n\to \infty,E\to \infty} \int_{\varSigma \cap[-E,E]} \sum_{j=1}^n \overline{ (U_j \Phi (L) f)(\lambda)} \langle \varphi_j (\lambda), v \rangle_d d\mu (\lambda) \\
\: \; & \stackrel{(Claim \; 1)}{=} & \int_{\varSigma} \sum_{j=1}^\infty  \overline{(U_j \Phi (L) f ) (\lambda) }  \langle \varphi_j (\lambda),  v \rangle_d d\mu (\lambda)\\
&=& \int_{\varSigma} \sum_{j=1}^\infty  \langle (U_j \Phi (L) f ) (\lambda) \varphi_j (\lambda), v \rangle_d d\mu (\lambda)\\
&=& \int_{\varSigma} (\sum_{j=1}^\infty  (U_j (\Phi (L)) f ) (\lambda) \varphi_j (\lambda), v )  d\mu (\lambda)\\
\;\: &\stackrel{\eqref{calculating-w}}{=}& \int_{\varSigma}  ( \Phi (\lambda)(W f) (\lambda), v) d\mu (\lambda).
\end{eqnarray*}
This finishes the proof.

\medskip

\noindent\textbf{Remarks.} Two comments  on the relationship of our proof to \cite{PSW}
are in order:

(a) On the technical level the starting ingredients to make the considerations of \cite{PSW} work is the choice of a function $\gamma$ and an operator $T$ such that $\gamma (L) T^{-1}$ is a Hilbert-Schmidt operator. In our  theorem we have shifted attention to $\mathcal{D}$. This  yields  the additional requirement that $\mathcal{D}$ must be contained in the domain of $T$. The upshot of this is that we can apply our theorem whenever $\gamma$, $T$, $\mathcal{D}$ are chosen such that $\mathcal{D}$ belongs to the domain of $T$ and $\gamma (L) T^{-1}$ is Hilbert-Schmidt (and $|\gamma|$ is strictly positive on the spectrum of $L$ and $T\geq 1$ holds). Depending on the situation one may then vary these parameters.

(b) The considerations of \cite{PSW} yield functions $\varphi_j (\lambda)$ in a (larger) Hilbert space than $\mathcal{H}$. The arguments above can be understood as providing a Hilbert space structure by declaring  these $\varphi_j (\lambda)$, $j\in J$, to be an orthonormal basis and then showing that this indeed works for almost all $\lambda$.

\medskip

The next corollary follows from the above arguments and \cite{PSW}.

\begin{coro}\label{cor-m} In the situation of  Theorem~\ref{main-abstract},
if $f\in D(L)\cap\mathcal{D}$ and $  Lf\in \mathcal{D}$ then
\begin{equation*}
\langle \varphi_{j} (\lambda), L f \rangle_d =\lambda \langle \varphi_{j} (\lambda), f \rangle_d .
\end{equation*}for
$\mu$ - almost every $\lambda$.
\end{coro}

\section{Some \emph{a priori} growth restrictions on generalized eigenfunctions} \label{growth}
In this section we discuss  a specific way of choosing $T$ by essentially requiring that  $T$ is a Hilbert-Schmidt operator with an orthonormal basis of eigenfunctions belonging to $\mathcal{D}$.  Then, the elements of $\mathcal{H}_\lambda$ appearing in the main theorem of the previous section can be seen to  satisfy some growth type restrictions.
For graphs this will have some direct applications.


Throughout this section we assume  that $\mathcal{H}$ be a separable Hilbert space and $L$ a  selfadjoint  operator on $\mathcal{H}$ with spectrum $\varSigma$.  Let $\mu$ be a spectral measure of $L$. Let $\mathcal{D}$ be a subspace of the domain of $L$, which is dense in  $\mathcal{H}$ and admits a countable algebraic base.
We now choose $T\geq 1$ such that $S = T^{-1}$ is Hilbert-Schmidt with all eigenfunctions belonging to $\mathcal{D}$. More specifically, we proceed as follows (compare above):
As $\mathcal{D}$  is dense in $\mathcal{H}$ it will contain  (by Gram-Schmidt procedure) an orthonormal basis $(v_n)_{n\in \N} $  of $\mathcal{H}$.  Fix now a map
$$\omega : \N \longrightarrow (0,\infty)\;\:\mbox{with}\;\: \sum \omega (n)^2 \leq 1.$$
Define $S : \mathcal{H}\longrightarrow \mathcal{H}$ to be the unique bounded operator with
$$S v_n = \omega (n) v_n$$
and $T = S^{-1}$.
Let $\mathcal{H}_- = \mathcal{H}_- (T)$ be the associated space. Then, $S$ is obviously a Hilbert Schmidt operator and hence so is $\gamma (L) S$ for any bounded $\gamma$. Choose $\gamma : \RR\longrightarrow \CC, \gamma  (s) = \frac{1}{ s+ i}$.

\begin{prop} Assume the situation just described. Then, for $\mu$-almost every $\lambda\in \varSigma$, the inclusion
$$\mathcal{H}_\lambda \subset \mathcal{H}_-$$
holds.
\end{prop}
\begin{proof}Via the orthonormal basis $v_n, n\in\NN$, it is possible to identify $\mathcal{H}$ with $\ell^2 =\ell^2 (\N)$. Then, $\mathcal{H}_+$ becomes $\{ f\in \ell^2 : \sum_n \frac{|f(n)|^2}{|\omega(n)|^2} < \infty\}$ and $\mathcal{H}_-$ becomes
$$\{f : \N\longrightarrow \C : \sum_n |\omega (n)|^2 |f(n)|^2 < \infty \}.$$
In this sense $\mathcal{H}_-$ consists of all those vectors with an expansion of the form $\sum f(n) v_n$ with $\sum_n |f(n)| |\omega (n)|^2 < \infty$.


For each of the elements $v_n$ of the orthonormal basis we have
$$ 1 = \|v_n\|^2 = \int_\varSigma \sum_j |\langle \varphi_j (\lambda),v_n\rangle|^2 d\mu.$$
This gives
$$1 \geq \sum_n |\omega (n)|^2 \|v_n\|^2  = \int_{\varSigma} \sum_{j,n} |\omega (n) \langle \varphi_j (\lambda), v_n\rangle_d|^2 d\mu (\lambda)$$
and we conclude that
$$ \sum_{j,n} |\omega (n) \langle \varphi_j (\lambda), v_n\rangle_d|^2 < \infty$$
for $\mu$-almost every $\lambda$. For each such $\lambda$ we then infer from Cauchy-Schwartz inequality
$$\sum_{n} |\omega (n)|^2 |\sum_j c_j \langle \varphi_j (\lambda), v_n\rangle_d|^2 \leq \sum_k |c_k|^2 \sum_{n,j} |\omega (n) \langle \varphi_j (\lambda), v_n\rangle_d|^2 < \infty$$
for any square summable sequence $(c_j)$. This directly  gives the desired statement by the discussion of the beginning of the proof.
\end{proof}

\section{Metric measure spaces and finer growth  properties} \label{Metric}

The proof of the  main abstract result has provided us with Hilbert spaces of generalized eigenfunctions together with   orthonormal bases  $\varphi_j (\lambda)$, $j\in J$ for each $\lambda$ in the spectrum.
In this section we study specific regularity features, viz  growth  properties and local regularity,  of these functions  $\varphi_j (\lambda)$.
Note that  these functions directly arise   from  the   the main result of  \cite{PSW} (see discussion at the end of Section \ref{general}).
 Therefore, all studies of regularity  properties of generalized eigenfunctions  based on \cite{PSW} can be directly applied in our setting. In this section, we discuss how some of the abstract ingredients of \cite{BdMS} can be applied and generalized. For the convenience of the reader and as regularity  is a crucial feature we present rather complete arguments.


Throughout this section we assume that
our separable Hilbert space is given as $\mathcal{H} =  L^2 (X,m)$
with a suitable measure space $(X,m)$.


\begin{theorem} \label{thm-bound}
Let $L$ be  a selfadjoint  operator on $L^2 (X,m)$ with spectrum $\varSigma$. Let $\gamma : \R\longrightarrow \R$ be continuous and bounded  with $|\gamma|>0$ on $\Sigma$ and  assume that the range of $\gamma (L)$ is contained in $L^\infty (X,m)$.
Let $w:X\to[1,\infty)$  be an arbitrary function such that  $w^{-1}\in L^2 (X,m)$. Then, there exists an expansion in generalized eigenfunctions such that for $\mu$-almost every $\lambda$ the basis  elements  $(\varphi_j (\lambda))_j$ consist of functions on $X$   satisfying
$$ w^{-1} \varphi_j (\lambda)\in L^2 (X,m)$$
for all $j\in J$.

Moreover, for any $f \in \mathcal{H}_+ = \{ f: f=w^{-1}g,\ g\in L^2
(X,m) \}  $,  we have  an analogue of the Fourier transform
 \begin{equation}
    \label{eq-star2}
    \langle \varphi_j (\lambda), f \rangle_d =
    \int_X  \varphi_j (\lambda)(x) f(x)      dm (x),
\end{equation}
which is an ordinary Lebesgue integral (and not an abstract
duality),  and  the inverse Fourier transform
\begin{equation*}
f(x)=\lim_{n\to N, E\to \infty} \sum\limits_{j=1}^n
\int\limits_{M_j\cap [-E,E]}
\langle \varphi_j (\lambda), f \rangle_d \;
\varphi_{j} (\lambda)(x) d\mu(\lambda).
\end{equation*}

In addition, we have a direct analogue of the Plancherel  formula
\begin{equation}
    \label{eq-star}
    \int_\varSigma \sum_{j\in J}   |
    \langle \varphi_j (\lambda), f \rangle_d
    |^2 d\mu (\lambda) = \|f\|^2 < \infty .
\end{equation}
\end{theorem}
\begin{proof} Let $T$ be  the operator of multiplication by $w$, and $S = T^{-1}$  the operator of multiplication by $w^{-1}$. Then
 $\gamma (L)  S$ is a Hilbert-Schmidt operator because of the
  Grothendieck factorization theorem, as $\gamma (L) $ maps $L^2$  into $L^\infty$ and multiplication by $w^{-1}$ maps $L^\infty$ into $L^2$.
The remaining statements are now a direct consequence of  Theorem~\ref{main-abstract} and its proof. In particular,  equation \eqref{eq-star}  follows from our Claim~1 in the proof of  Theorem~\ref{main-abstract}.
\end{proof}

\noindent\textbf{Remarks.}
(a)  Often the main  interest lies in  situations where the semigroup $e^{-t L}$, $t>0$, maps $L^2$ into $L^\infty$.
  This is studied under the heading  of ultracontractivity, see \cite{Dav,Sim,SV} for further discussion and references,
  mostly dealing with Schr\"odinger operators  or, more generally, suitable perturbations of Dirichlet forms.
  Of course, in such a situation we may take $\gamma : [0,\infty)\longrightarrow \RR$ to be $\gamma (s) = e^{-t s}$ for any $t>0$.

(b) The first two parts  of the previous  theorem are  essentially a slight reformulation of   main abstract ingredients of \cite{BdMS} (which we place in a somewhat more general context of operators on $L^2 (X,m)$, rather than Dirichlet forms). The proof via Grothendieck factorization is taken from \cite{BdMS} (which  in turn is inspired by  \cite{Stol}).

 \medskip

The main emphasis of \cite{BdMS} is on local regular
 Dirichlet forms that satisfy certain subexponential volume growth conditions.
 However it is noted in \cite{BdMS} that these conditions can be separated.
 In particular,
 Theorem~\ref{thm-growth} shows that  subexponential (or other) volume growth conditions
 can be proved for generalized eigenfunctions without assuming that we have a Dirichlet form.

For Theorem~\ref{thm-growth}
we
will  make the further assumption that  $X$ is a metric space with a metric $\varrho$ (and the $\sigma$-algebra generated by the topology).  We then refer to $(X,m,\varrho)$ as a metric measure space.  The closed ball around a point $x\in X$ with radius $R$ will be denoted by $B(x,R)$.
In this context, we can study  subexponentially bounded
eigenfunctions defined as follows.

\begin{definition}
Let $(X,m,\varrho)$ be a metric measure space. A function $f$ on
$(X,m,\varrho)$ is said to be subexponentially bounded in $ L^2$
sense if for some $x_0\in X$ and  $\omega (x) = \varrho (x_0,x)$ the
function $e^{-\alpha \omega} f$ belongs to  $ L^2
  (X,m)$ for any $\alpha >0$.
\end{definition}




The abstract core of the corresponding argument of \cite{BdMS}  then gives the following.

\begin{theorem}\label{thm-growth}
Let $(X,m,\varrho)$ be a metric measure space. Let $L$ be  a
selfadjoint  operator on $L^2 (X,m)$ with spectrum $\varSigma$. Let
$\gamma : \R\longrightarrow \R$ be continuous and bounded  with
$|\gamma|>0$ on $\Sigma$.   Assume that the following holds:
\begin{itemize}

\item The range of $\gamma (L)$ is contained in $L^\infty (X,m)$.

\item For all $x\in X$, $R>0$ we have $m(B(x,R))<\infty$, and  $\lim\limits_{R\to\infty}e^{-\alpha\cdot R}m(B(x,R))= 0$   for any $\alpha>0$.

\end{itemize}
Then, there exists an expansion in generalized eigenfunctions such that for $\mu$-almost every $\lambda$ the basis  elements  $(\varphi_j (\lambda))_j$ consist  of subexponentially bounded functions  on $X$.
\end{theorem}
\begin{proof} Fix $x_0\in X$ and define $\omega (y) := \varrho (x_0,y)$.  By the second assumption the functions $w^{-1}_\alpha  = e^{- \alpha \omega}$ belong to $L^2 (X,m)$ for any $\alpha >0$. Then, the first assumption and the first result of this section, Theorem~\ref{thm-bound},  give that for $\mu$-almost every $\lambda$ the function  $w^{-1}_\alpha \varphi_j (\lambda)$ belongs to $L^2 (X,m)$ for  every $j\in J$.  Appealing to a countable dense subset of $\alpha$'s in $(0,\infty)$ we now obtain  the statement.
\end{proof}

\noindent\textbf{Remarks.} (a)  Under suitable assumptions it is possible to show a converse to the previous  theorem i.e. every $\lambda$ admitting an subexponentially bounded generalized eigenfunction must then belong to the spectrum of $L$. This type of result is known as Shnol theorem (see \cite{BdMLS, FLW, LSS,LSV} for recent   results of this type for operators arising from Dirichlet forms and further references).

(b) Note that  the first   theorem of this section  can be slightly generalized: If there is a measurable  bounded non-vanishing  function $b$ on $X$  and $b \gamma (L)$ maps into $L^\infty$, then with the  operator $T$ of multiplication by  $w b^{-1}$ we have
 $S =  w^{-1} b$ and $S \gamma (L)$ is Hilbert-Schmidt. Hence, the statement of Theorem \ref{thm-bound} (and its proof) carry over with  $w$ replaced by $w b^{-1}$. This then yields a corresponding version of the preceding theorem as well.

(c) The above considerations do not  give growth restrictions on all generalized eigenfunctions in  the $\mathcal{H}_\lambda$ but only on a basis of this space. This may  be considered a weakness.  However, there is no reason in general why all elements in these spaces should satisfy strong  growth restrictions (note that this spaces can be infinite-dimensional, such as in \cite{T98}).

(d) One can apply the same methods to other (e.g. polynomial)
growth conditions on the measure $m(B(x,R))$ of the balls
(this remark is due to Peter Kuchment, as an extension of the classical ideas for periodic operators and quantum graphs, see \cite{BK,K93}).

\medskip

Theorem~\ref{thm-growth} can be compared and contrasted with
 Corollary~\ref{cor-m}, which describes `local' properties of generalized eigenfunctions (e.g. for local Dirichlet forms) without explicitly assuming volume growth conditions. We are interested in these questions because the framework of Theorem~\ref{thm-bound} fits well in the general ideas of   \cite{St89,St12}  (see also
 \cite[(1.1)-(1.2) and (3.1)-(3.5)]{ST} and \cite{DRS,R12,RST}), which apply to an wide array spaces of exponential growth such as hyperbolic groups, symmetric spaces   and fractafolds.

 In particular, a natural question to ask  is whether in some sense
\begin{equation}
    L \varphi_j (\lambda) =  \lambda \varphi_j   (\lambda)
\end{equation}
 but, since $\varphi_j (\lambda)(x)$ are only locally square integrable, this question is not well posed
 in terms of the usual spectral theory of self-adjoint operators. However, it is   natural to discuss this question
 if  $(X,m,\varrho)$ is a locally compact complete  metric measure space, and
$m$ is a Radon measure, which implies in particular that $m(B(x,R))<\infty$, without specific volume growth conditions.
Furthermore,
 it is natural to assume that
$L$ is the generator of a regular strongly local Dirichlet form $(\ee,D(\ee))$ on  $L^2(X,m)$,
and that
the set  $\mathcal{D}$ contains a dense set, in $L^2(X,m)$ and in $C_0(X,\varrho)$, of  compactly supported continuous
functions  in the domain $D(\ee)$ of $\ee$. Note that such a set $\Dee$ always exist by the  standard theory of regular Dirichlet forms.
Also we can choose $w$ in Theorem~\ref{thm-bound} to be continuous and thus locally bounded.


In this situation    $f\in L^2_{loc}(X,m)$
is of locally finite energy, i.e. belongs to $\Dl(\ee)$, if for any ball
$B(x,R)$ there is $u\in D(\ee)$ which coincides with $f$ on $B(x,R)$.
Note that if $f\in L^2_{loc}(X,m)$ and $v\in\Dee$ then $\ee(f,v)$ is well defined as $\ee(f,v)=\ee(u,v)$
for any $u\in D(\ee)$ which coincides with $f$ on $B(x,R)$ provided that  $\text{supp}(v)\subset B(x,R)$.
 Furthermore,
we can use the following definitions.
\begin{definition} \label{df-loc}
\begin{enumerate}
\item
 We say that $f\in   \Dl(\ee)$ belongs to the weak local domain of $L$  if
there exists $g\in L^2_{loc}(X,m)$
such that $$\ee(f,v)=(g,v)_{L^2(X,m)}  $$ for any $v\in\Dee$.
In this case we will say that $f\in   \Dl(L)$ and $$Lf=g$$ in $  L^2_{loc}(X,m)$.
\item
We say that $f\in   \Dl(L)$ belongs to the strong local domain of $L$
if for any ball
$B(x,R)$ there is $u\in D(L)$
which coincides with $f$ on $B(x,R)$. Note that such an $f$ also belongs to the weak local domain of $L$, and so $Lf=g$ for some $ g\in L^2_{loc}(X,m)$.
\end{enumerate}\end{definition}
To the best of our knowledge these notions are not well studied in the context of  general Dirichlet forms. The weak domain of the Laplacian appears implicitly in \cite[Theorem 1(b)]{PSW} and \cite[Theorem 2.1(s)]{BdMS}. The strong domain of the Laplacian  is studied in detail
 in a more restricted
context of analysis on finitely ramified fractals, see \cite{BST,DRS,RST,St00t} and references therein.

The following theorem
holds for infinite Sierpinski fractafolds, which can be described as an at most countable union of
isometric
copies of the standard Sierpinski gasket such that
\begin{enumerate}
    \item these copies do not intersect except for corner points;
    \item each corner point is contained in at most two of these copies.
\end{enumerate}
One can see that every point of a Sierpinski fractafold has a neighborhood homeomorphic to a neighborhood in the standard
Sierpinski gasket, which explains the term `fractafold' introduced by Strichartz in  \cite{St99n,St03}, similarly to the notion of a manifold  where
 every point has a neighborhood homeomorphic to a neighborhood in the standard Euclidean space.

The standard Laplacian is uniquely defined on a   Sierpinski fractafold and is a local operator.
Moreover, according to \cite{RST},
there exits  a  set  $\Dee$ (with countable algebraic basis) which is dense set in $L^2(X,m)$ and in $C_0(X,\varrho)$  and consists of compactly supported functions in the domain of any positive power $L^n$ of $L$.

\begin{theorem}\label{thm-local}
If $L$ is the standard Laplacian on  a Sierpinski fractafold $X$, as defined in \cite{ST}, then
the generalized eigenfunctions $\varphi_j (\lambda)$
belongs to the strong local domain of $L$
for $\mu$-almost every $\lambda$. In particular, they have a continuous version for which the point-wise approximating  formula holds
$$
    L\varphi_j (\lambda) (x)= \lim\limits_ {n{\to}\infty} 5^ n
\Delta_n \varphi_j (\lambda)(x)=\lambda\varphi_j (\lambda) (x)
$$
where $\Delta_n $ is the standard graph Laplacian on the graphs $(V_n,E_n)$ approximating the Sierpinski gasket and $x$ is any vertex in these graphs, see Figure~\ref{fig-Vn}.
\end{theorem}

\begin{figure}
\centerline
{
\def\Tri#1#2#3#4{{
\put(#1,#2){\line(3,5){#3}}
\put(#1,#2){\line(3,0){#4}}
\count223=#1
\advance\count223 by #4
\put(\count223,#2){\line(-3,5){#3}}}}
\def\Trii#1#2#3#4#5{{
\put(#1,#2){\circle*{#5}}
\count223=#1
\advance\count223 by #4
\put(\count223,#2){\circle*{#5}}
\count226=#1
\advance\count226 by #3
\count224=#3
\divide\count224 by 3
\multiply\count224 by 5
\count225=#2
\advance\count225 by \count224
\put(\count226,\count225){\circle*{#5}}
}}
\def\Triii#1#2#3#4#5{{
\put(#1,#2){\circle{#5}}
\count223=#1
\advance\count223 by #4
\put(\count223,#2){\circle{#5}}
\count226=#1
\advance\count226 by #3
\count224=#3
\divide\count224 by 3
\multiply\count224 by 5
\count225=#2
\advance\count225 by \count224
\put(\count226,\count225){\circle{#5}}
}}
\def\Spic#1#2#3#4{{
\count205=#2\count206=#2\count207=#2\count208=#2\count209=#2
\count210=#1\count211=#3\count214=#3\count212=#4
\divide\count210 by 2
{\ifnum\count210>0
\Spic{\count210}{#2}{#3}{#4}
\multiply\count207 by 3
\multiply\count208 by 6
\multiply\count209 by 5
\multiply\count207 by \count210
\multiply\count208 by \count210
\multiply\count209 by \count210
{\advance\count211 by \count207
\advance\count212 by \count209
\Spic{\count210}{#2}{\count211}{\count212}}
{\advance\count214 by \count208
\Spic{\count210}{#2}{\count214}{#4}}
\else
\multiply\count205 by 3
\multiply\count206 by 6
\Tri{#3}{#4}{\count205}{\count206}
\fi
}}}
\def\Spicc#1#2#3#4#5{{
\count205=#2\count206=#2\count207=#2\count208=#2\count209=#2
\count210=#1\count211=#3\count214=#3\count212=#4
\divide\count210 by 2
{\ifnum\count210>0
\Spicc{\count210}{#2}{#3}{#4}{#5}
\multiply\count207 by 3
\multiply\count208 by 6
\multiply\count209 by 5
\multiply\count207 by \count210
\multiply\count208 by \count210
\multiply\count209 by \count210
{\advance\count211 by \count207
\advance\count212 by \count209
\Spicc{\count210}{#2}{\count211}{\count212}{#5}}
{\advance\count214 by \count208
\Spicc{\count210}{#2}{\count214}{#4}{#5}}
\else
\multiply\count205 by 3
\multiply\count206 by 6
\Trii{#3}{#4}{\count205}{\count206}{#5}
\fi
}}}
\def\Spiccc#1#2#3#4#5{{
\count205=#2\count206=#2\count207=#2\count208=#2\count209=#2
\count210=#1\count211=#3\count214=#3\count212=#4
\divide\count210 by 2
{\ifnum\count210>0
\Spiccc{\count210}{#2}{#3}{#4}{#5}
\multiply\count207 by 3
\multiply\count208 by 6
\multiply\count209 by 5
\multiply\count207 by \count210
\multiply\count208 by \count210
\multiply\count209 by \count210
{\advance\count211 by \count207
\advance\count212 by \count209
\Spiccc{\count210}{#2}{\count211}{\count212}{#5}}
{\advance\count214 by \count208
\Spiccc{\count210}{#2}{\count214}{#4}{#5}}
\else
\multiply\count205 by 3
\multiply\count206 by 6
\Triii{#3}{#4}{\count205}{\count206}{#5}
\fi
}}}
%
%
\begin{picture}(60,50)(0,0) \setlength{\unitlength}{0.5pt}\small\thinlines
\put(-20,70){$V_{0}:$}
{\def\LLL{6pt}
\setlength{\unitlength}{10pt}
\Spicc{1}{1}{0}{0}{.5}%
\Spic{1}{1}{0}{0}%
}
\end{picture}\ \ \
\ \ \
\ \ \
\begin{picture}(60,50)(0,0) \setlength{\unitlength}{0.5pt}\small\thinlines
\put(-20,70){$V_{1}:$}
{\def\LLL{6pt}
\setlength{\unitlength}{5pt}
\Spicc{2}{1}{0}{0}{.9}%
\Spic{2}{1}{0}{0}%
}
\end{picture}\ \ \
\ \ \
\ \ \
\begin{picture}(60,50)(0,0) \setlength{\unitlength}{0.5pt}\small\thinlines
\put(-20,70){$V_{2}:$}
{\def\LLL{6pt}
\setlength{\unitlength}{2.5pt}
\Spicc{4}{1}{0}{0}{1.5}%
\Spic{4}{1}{0}{0}%
}
\end{picture}%
}
\caption{Discrete graph approximations to the standard Sierpinski gasket.}\label{fig-Vn}

\end{figure}
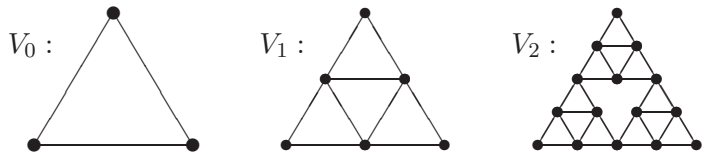

\begin{proof}
The proof of this theorem, which we outline briefly, is based on  the
detailed technical local analysis of the functions in the domain of $L$
available in the cited literature, in particular in \cite{St06book}.
\begin{enumerate}
 \item  $X$ is a countable union of compact  fractafolds $F_k$, each having a finite boundary
 $\partial F_k$. In what follows we  fix $k$.
 \item  restriction of all functions to $F_k$  defines a Dirichlet form on $L^2(F_k,m)$ which we will denote by the same notation $\ee$.
 \item  there is a continuous Green's function $g_k(x,y)$ on $F_k$ which is the integral kernel of the
 Green's operator $G_k$. This operator is the inverse of the $L_k$, which is defined as the restriction of $L$ to the functions with support in $F_k$ or, alternatively, $L_k$ is the unique Dirichlet Laplacian on $F_k$.
 \item  fix $\lambda>0$ for which \eqref{eq-duaL} holds for all $u\in\Dee$ and define $f_k=\lambda G_k\left(\varphi_j (\lambda)\big|_{{\vphantom{f}}_{F_k}}\right)$. We have $L_kf_k=\lambda \varphi_j (\lambda)$ on $F_k$.

 \item  define $h_k=\left(\varphi_j (\lambda)-f_k\right)\big|_{{\vphantom{f}}_{F_k}}$ and observe that
 $$(h_k,Lu)=0$$ for any $u\in\Dee$ with support in the interior of $F_k$. This implies the key observation, which follows from \cite[Theorem 4.5]{SU}, that $h_k$ has a continuous version on $F_k$ which is a harmonic function   in the interior of $F_k$.

 \item  we have that $\varphi_j (\lambda)=f_k+h_k$ in  $L^2(F_k,m)$, which implies that
 $\varphi_j (\lambda)$ has a continuous version.

 \item  it follows from \cite{RST} that this version is in the strong local domain of $L$ and the point-wise approximating  formula.
 \end{enumerate}
\end{proof}

\noindent\textbf{Remark.} It is straightforward to generalize this theorem to the case of infinite finitely ramified cell structures, as defined in \cite{T08}, and it would be   interesting to understand under which assumptions it is true for general local regular Dirichlet forms, which will be subject of future work. The steps of the proof outlined above may be applicable for resistance forms of Kigami \cite{Ki1,Ki2} using the methods developed in \cite{BdMS,HKT2}.

\section{Getting back to discrete measure spaces - Proof of Theorem \ref{main-concrete}}\label{Getting}
In this section we consider a discrete measure space $(V,m)$ and  provide a proof of Theorem \ref{main-concrete}. To do so, we first discuss two  basic results on selfadjoint  operators on $\ell^2 (V,m)$. While the results are rather simple they may be of independent interest. In our context they will be needed to apply the abstract theorem to the concrete situation. Throughout this section we assume that we are given a discrete measure space $(V,m)$ as defined in the first section.

\medskip

We first  show that for locally finite selfadjoint operators on $\ell^2 (V,m)$  the two notions of $C_c (V)$-eigenfunction and of generalized eigenfunctions introduced previously agree. Here, we identify  the algebraic dual of $C_c (V)$ with $C(V)$ via the pairing $(\cdot,\cdot)_m$ introduced in Section \ref{Locally}. More specifically any $g \in C(V)$ produces a functional $C_c (V)\longrightarrow \CC$ via
$$(g,u)_m = \sum_{x\in V} \overline{g (x)} u(x) m(x)$$
for $u\in C_c (V)$.

\medskip

\begin{lemma}\label{characterization-generalized-ef} Let $L$ be a locally finite selfadjoint operator with kernel $l$. Then, the following assertions are equivalent for $\varphi \in C(V)$ and $\lambda\in \R$:
\begin{itemize}
\item[(i)] $\varphi$ is a generalized eigenfunction of $L$ to the eigenvalue $\lambda$.
\item[(ii)]  $\varphi$ is a $C_c (V)$-eigenfunction  of $L$ to the eigenvalue $\lambda$.
\end{itemize}
\end{lemma}
\begin{proof} By local finiteness of $L$ we have $D(L) = C(V)$ and $\varphi$ is a generalized eigenfunction to $\lambda$ if and only if $(\widetilde{L} - \lambda ) \varphi = 0$ holds. On the other hand, by the dual pairing $(\cdot,\cdot)_m$ the function  $\varphi$  is a $C_c (V)$-eigenfunction to $\lambda$ if and only if
$$ \sum \overline{\varphi (x)} (L - \lambda) v (x) m(x) = 0$$
holds for all $v\in C_c (V)$. Now,
   a short calculation  shows that
   \begin{eqnarray*}
\sum_{x} \overline{ (\widetilde{L} - \lambda) w } (x)  v(x)   m(x) & = &  \sum_{x,y}  \overline{ (l(x,y) - \lambda) w(y)}  m(y)  v(x)   m(x)\\
 & = &   \sum_{y\in V} \overline{w (y)}  {(L - \lambda) v} (y)  m(y)
 \end{eqnarray*}
for all $w\in C(V)$ and $v\in C_c (V)$. Here, all sums are absolutely convergent. This easily gives the statement of the lemma.
\end{proof}

After this preparation we can now  provide a \textit{Proof of Theorem \ref{main-concrete}}.
We can choose $\mathcal{D} := C_c (V)$. This space has indeed a countable algebraic basis as  $V$ is  is countable. Then,
Lemma \ref{characterization-generalized-ef}  shows that any $\mathcal{D}$-eigenfunction of $L$ is indeed a generalized eigenfunction of $L$. Then, Theorem \ref{main-abstract} gives all statements of Theorem \ref{main-concrete} up to the last one. The last statement follows from   the considerations of Section \ref{growth}. This finishes the proof of the theorem.

\medskip

\noindent\textbf{Remark.}
 In order to be specific in the proof of Theorem \ref{main-concrete} one may make the following choices:   Let
$$\omega : V\longrightarrow (0,\infty),\: \mbox{with}\;\: \sum_{x\in V} \omega (x)^2 \leq 1$$
 be given  and $S  = M_\omega$ be the operator of multiplication by $\omega$.
 Then $S$ is a Hilbert-Schmidt operator (compare discussion in Section \ref{growth}).  Let
 $$T = M_{\frac{1}{\omega}}.$$
Then, $T$ is selfadjoint with $T\geq 1$ (as $\omega$ is real valued with $0< \omega \leq  1$) and
$$ T^{-1} = S.$$
The associated Hilbert spaces can be described explicitly as
$$\hil_+:= \hil_+(T) :=D(T) =\{ u : V\longrightarrow \CC : \sum_{x\in V} \frac{ |u(x)|^2}{\omega (x)^2} m(x)  < \infty\}$$
and
$$ \hil_- = \{ u :V\longrightarrow \CC : \sum \omega(x)^2 |u(x)|^2 m(x)  < \infty\}.$$
The corresponding dual pairing  between $\hil_-$ and $\hil_+$ is given by
$$\langle v, u \rangle_d = \sum_{x\in V}  \overline{v(x)}  u(x) m(x) $$
for $u\in\hil_+, v\in \hil_-$.
Note that $\hil_+$ is contained in $\ell^2 (V,m)$ as $0< \omega \leq 1$. Note also that $\langle u,v\rangle_d$ given as above is indeed well defined for $u\in \hil_+$ and $v\in \hil_-$ by the H\"older inequality.
Let now  $$\gamma : [0,\infty) \longrightarrow \R, s \mapsto ( s + i)^{-1} .$$
Then, $\gamma$ is obviously continuous and does not vanish on the positive half-axis.
Then, $ \gamma (L) T^{-1} = (L + i)^{ -1}  S $ is Hilbert-Schmidt as $S$ is Hilbert-Schmidt and $(L+ i)^{-1}$ is bounded.

\appendix

\section{Direct integrals and measurable families of Hilbert spaces}\label{Direct}
In this section we briefly recall some direct integral theory. For further details we refer to e.g.  \cite{Dix, Nie}.

\medskip

Let $(\varSigma, \mu)$ be a measure space.  A family  of Hilbert spaces $(\mathcal{H}_\lambda, \langle \cdot,\cdot\rangle_\lambda)$, $\lambda\in \varSigma$, together with a countable collection $(w_n)_{n\in \NN}$ of functions on $\varSigma$ with $w_n (\lambda)\in \mathcal{H}_\lambda$ for each $\lambda\in \varSigma$, is called measurable family over $(\varSigma,\mu)$ if the following two properties hold:
\begin{itemize}
\item For each $\lambda\in\varSigma$ the linear hull of $\{w_n (\lambda): n\in\NN\}$ is dense in $\mathcal{H}_\lambda$.

\item  For all $n,m\in\NN$, the function $\varSigma \longrightarrow \CC$, $\lambda\mapsto \langle w_n (\lambda), w_m (\lambda)\rangle_\lambda$, is measurable.
\end{itemize}
Given such a measurable family of Hilbert spaces  we call a function $w $ on $\varSigma$ with $w(\lambda)\in \mathcal{H}_\lambda$ for each $\lambda\in\varSigma$ measurable if
$$\varSigma \longrightarrow \CC, \;\lambda\mapsto \langle w (\lambda), w_n (\lambda)\rangle_\lambda,$$
is measurable for any $n\in \NN$.

Then, the vector space  $\mathcal{L}^2  (\varSigma, (\mathcal{H}_\lambda))$ consisting of all measurable functions $w$ on $\varSigma$ with $w(\lambda)\in \mathcal{H}_\lambda$ for each $\lambda\in \varSigma$
such that
$$\int_\varSigma \|w (\lambda)\|^2 d\mu (\lambda)< \infty $$
carries a semi-scalar product given by
$$\langle w, v\rangle := \int_\varSigma \langle w (\lambda),v(\lambda)\rangle_\lambda d\mu (\lambda).$$
The quotient of $\mathcal{L}^2  (\varSigma, (\mathcal{H}_\lambda))$  by the subspace $\mathcal{N}$ consisting of all elements $w$ with $\langle w, w\rangle = 0$ is a Hilbert space and denoted by $\int_\varSigma^\oplus \mathcal{H}_\lambda d\mu (\lambda)$.

Let   a direct integral Hilbert space  $ \mathcal{K} =\int_\varSigma^\oplus \mathcal{H}_\lambda d\mu (\lambda)$ be given. Then, a bounded  operator $A$ on $\mathcal{K}$ is called decomposable if there exists a family of bounded operators $A_\lambda: \mathcal{H}_\lambda\longrightarrow \mathcal{H}_\lambda$, $\lambda\in \varSigma$,
with
$$(A f) (\lambda) =  A_\lambda f(\lambda)$$
for $\mu$-almost every $\lambda$
for each  $f\in \mathcal{K}$. This is then written as
$$A = \int_\varSigma^\oplus A_\lambda d\mu (\lambda).$$
If for such an $A$ there exists an measurable function $F : \varSigma \longrightarrow \CC$ with
$$A_\lambda f =  F (\lambda)f\;\:\mbox{for all $\lambda\in \varSigma$ and $f\in \mathcal{H}_\lambda$}$$
one also writes $A = M_F$ (by a slight abuse of language).

If $A$ is a  (not necessarily bounded) selfadjoint operator on $\mathcal{K}$ and  $(A_\lambda)$ is a family of (not necessarily bounded) selfadjoint  operators on the respective $\mathcal{H}_\lambda$  we write
$$ A = \int_\varSigma^\oplus A_\lambda d\mu (\lambda)$$
if and only if
$$\Phi (A)  = \int_\varSigma^\oplus \Phi (A_\lambda) d\mu (\lambda)$$
for all bounded measurable $\Phi: \R\longrightarrow \CC$.

\def\arxiv#1{\href{http://arxiv.org/abs/#1}{arXiv:#1}}

\end{document}